\documentclass[leqno]{amsart}
\usepackage{amsmath,amsthm,amsfonts,amssymb}

\numberwithin{equation}{section}
\theoremstyle{plain}

\newtheorem{theorem}{Theorem}

\newtheorem{proposition}[theorem]{Proposition}

\newtheorem{remark}[theorem]{Remark}

\begin{document}

\title[A CIR Process with Hawkes Jumps]{Limit Theorems for a Cox-Ingersoll-Ross Process with Hawkes Jumps}

\author{LINGJIONG ZHU}
\address
{Courant Institute of Mathematical Sciences\newline
\indent New York University\newline
\indent 251 Mercer Street\newline
\indent New York, NY-10012\newline
\indent United States of America}
\email{ling@cims.nyu.edu}

\date{19 April 2013. \textit{Revised:} 28 August 2013}
\subjclass[2000]{60G07, 60G55, 60F05,60F10.}
\keywords{Cox-Ingersoll-Ross process, point processes, Hawkes processes, self-exciting processes, 
central limit theorem, large deviations.}

\begin{abstract}
In this paper, we propose a stochastic process, which is a Cox-Ingersoll-Ross process with Hawkes jumps.
It can be seen as a generalization of the classical Cox-Ingersoll-Ross process and the classical Hawkes process
with exponential exciting function. Our model is a special case of the affine point processes. Laplace transforms
and limit theorems have been obtained, including law of large numbers, central limit theorems and large deviations.
\end{abstract}

\maketitle

\section{Introduction and Main Results}

\subsection{Cox-Ingersoll-Ross Process}

A Cox-Ingersoll-Ross process is a stochastic process $r_{t}$ satisfying the following stochastic differential equation,
\begin{equation}
dr_{t}=b(c-r_{t})dt+\sigma\sqrt{r_{t}}dW_{t},
\end{equation}
where $W_{t}$ is a standard Brownian motion, $b,c,\sigma>0$ and $2bc\geq\sigma^{2}$. The process is proposed by Cox, Ingersoll and Ross
in Cox et al. \cite{Cox} to model the short term interest rate. Under the assumption $2bc\geq\sigma^{2}$, Feller \cite{Feller}
proved that the process is non-negative. Given $r_{0}$, it is well known that $\frac{4b}{\sigma^{2}(1-e^{-bt})}r_{t}$
follows a noncentral $\chi^{2}$ distribution with degree of freedom $\frac{4bc}{\sigma^{2}}$ and 
non-centrality parameter $\frac{4b}{\sigma^{2}(1-e^{-bt})}r_{0}e^{-bt}$. As $t\rightarrow\infty$, $r_{t}\rightarrow r_{\infty}$,
where $r_{\infty}$ follows a Gamma distribution with shape parameter $\frac{2bc}{\sigma^{2}}$ and scale parameter $\frac{\sigma^{2}}{2b}$.
The conditional first and second moments are given by, $s>t$,
\begin{align}
&\mathbb{E}[r_{s}|r_{t}]=r_{t}e^{-b(s-t)}+c(1-e^{-b(s-t)})\label{FirstMomentCIR}
\\
&\mathbb{E}[r_{s}^{2}|r_{t}]
=r_{t}\left(2c+\frac{\sigma^{2}}{b}\right)e^{-b(s-t)}+\left(r_{t}^{2}-r_{t}\frac{\sigma^{2}}{b}-2r_{t}c\right)e^{-2b(s-t)}\label{SecondMomentCIR}
\\
&\qquad\qquad\qquad
+\left(\frac{c\sigma^{2}}{2b}+c^{2}\right)\left(1-e^{-b(s-t)}\right)^{2}.\nonumber
\end{align}

The Cox-Ingersoll-Ross process has been widely applied in finance, mostly in short term interest rate, see e.g.  Cox et al. \cite{Cox}
and the Heston stochastic volatility model, see e.g.  Heston \cite{Heston}.
Other applications include the modelling of mortality intensities, see e.g.  extended Cox-Ingersoll-Ross process used by
Dahl \cite{Dahl} and of default intensities in credit risk models, 
see e.g.  as a special case of affine process by Duffie \cite{Duffie}.

A natural generalization of the classical Cox-Ingersoll-Ross process takes into account the jumps, i.e.
\begin{equation}
dr_{t}=b(c-r_{t})dt+\sigma\sqrt{r_{t}}dW_{t}+adN_{t},
\end{equation}
where $N_{t}$ is a homogeneous Poisson process with constant intensity $\lambda>0$. But in the real world, the occurence of events
may not be time-homogeneous and it should have dependence over time. 
Errais et al. \cite{Errais} pointed out ``The collapse of Lehman Brothers brought the financial system to the brink of a breakdown.
The dramatic repercussions point to the exisence of feedback phenomena that are channeled through the complex web of informational and contractual
relationships in the economy... This and related episodes motivate the design of models of correlated default timing that incorporate
the feedback phenomena that plague credit markets.''
According to Kou and Peng \cite{Kou}, ``We need better models to incorporate the default clustering effect, i.e., one default
event tends to trigger more default...''

In this respect, it is natural to replace Poisson process
by a simple point process which can describe the time dependence in a natural way. 
The Hawkes process, a simple point process that has self-exciting property and clustering effect becomes
a natural choice.

\subsection{Hawkes Process}

A Hawkes process is a simple point process $N$ admitting an intensity
\begin{equation}
\lambda_{t}:=\lambda\left(\int_{-\infty}^{t}h(t-s)N(ds)\right),\label{Hawkesdynamics}
\end{equation}
where $\lambda(\cdot):\mathbb{R}^{+}\rightarrow\mathbb{R}^{+}$ is locally integrable, left continuous, 
$h(\cdot):\mathbb{R}^{+}\rightarrow\mathbb{R}^{+}$ and
we always assume that $\Vert h\Vert_{L^{1}}=\int_{0}^{\infty}h(t)dt<\infty$. 
In \eqref{Hawkesdynamics}, $\int_{-\infty}^{t}h(t-s)N(ds)$ stands for $\int_{(-\infty,t)}h(t-s)N(ds)=\sum_{\tau<t}h(t-\tau)$, where
$\tau$ are the occurences of the points before time $t$.

In the literature, $h(\cdot)$ and $\lambda(\cdot)$ are often referred to
as exciting function and rate function respectively. 
An important observation is that a Hawkes process is Markovian if and only if $h(\cdot)$ is an exponential function.
One usually assumes that $\lambda(\cdot)$ is increasing and $h(\cdot)$ is decreasing.

A Hawkes process is linear if $\lambda(\cdot)$ is linear and it is nonlinear otherwise.
Linear Hawkes process, i.e. the classical Hawkes process, is named after Hawkes, who first invented the model
in Hawkes \cite{Hawkes}. Nonlinear Hawkes process was first introduced by Br\'{e}maud and Massouli\'{e} \cite{Bremaud}.

By the definition of Hawkes process, it has the self-exciting property, i.e. the intensity $\lambda_{t}$ increases
when you witness a jump. It therefore creates a clustering effect, which is to model the default clustering in finance.
When you do not witness new jumps, the intensity $\lambda_{t}$ decreases as $h(\cdot)$ decays.

The law of large numbers and central limit theorems for linear Hawkes process have 
been obtained in Hawkes and Oakes \cite{HawkesII}. 
The law of large numbers and central limit theorem have also been studied in Bacry et al. \cite{Bacry}
as a special case of multivariate Hawkes processes. 
The large deviation principle for linear Hawkes process was obtained in Bordenave and Torrisi \cite{Bordenave}.
The moderate deviation principle for linear Hawkes process
was obtained in Zhu \cite{ZhuIV}. For nonlinear Hawkes process, the central limit thereom was obtained
in Zhu \cite{ZhuIII} and the large deviations have been studied in Zhu \cite{ZhuI} and Zhu \cite{ZhuII}.

The central limit theorem of Hawkes process has been applied to study the high frequency trading and the microstructure in finance, 
see e.g. Bacry et al. \cite{Bacry} and Bacry et al. \cite{BacryII} and the large deviations result has 
been applied to study the ruin probabilities in insurance, see e.g. Stabile and Torrisi \cite{Stabile} and Zhu \cite{ZhuV}.

\subsection{A Cox-Ingersoll-Ross Process with Hawkes Jumps}

In this paper, we propose a stochastic process $r_{t}$ that satisfies the following stochastic differential equation,
\begin{equation}
dr_{t}=b(c-r_{t})dt+adN_{t}+\sigma\sqrt{r_{t}}dW_{t},\label{dynamics}
\end{equation}
where $W_{t}$ is a standard Brownian motion and $N_{t}$ is a simple point
process with intensity $\lambda_{t}:=\alpha+\beta r_{t}$ at time $t$.
We assume that $a,b,c,\alpha,\beta,\sigma>0$ and
\begin{itemize}
\item
$b>a\beta$. This condition is needed to guarantee that there exists a unique stationary process $r_{\infty}$
which satisfies the dynamics \eqref{dynamics}.

\item
$2bc\geq\sigma^{2}$. This condition is needed to guarantee that $r_{t}\geq 0$ with probability $1$.
Indeed, we know that $r_{t}$ stochastically dominates the classical Cox-Ingersoll-Ross process and hence $2bc\geq\sigma^{2}$
is enough to guarantee $r_{t}\geq 0$. On the other hand, on any given time interval, the probability
that there is no jump is always positive, which implies that $2bc\geq\sigma^{2}$ is needed to guarantee positivity.
\end{itemize}

The Cox-Ingersoll-Ross process with Hawkes jumps preserves the mean-reverting and non-negative properties
of the classical Cox-Intersoll-Ross process. In addition, it contains the Hawkes jumps, which have the self-exciting property
create a clustering effect.

Clearly, the Cox-Ingersoll-Ross process we proposed in \eqref{dynamics} includes the classical Cox-Ingersoll-Ross process
and the classical linear Hawkes process with exponential exciting function. We summarize this in the following.

\begin{enumerate}
\item
When $a=0$ or $\alpha=\beta=0$, it reduces to the classical Cox-Ingersoll-Ross process, i.e.
\begin{equation*}
dr_{t}=b(c-r_{t})dt+\sigma\sqrt{r_{t}}dW_{t}.
\end{equation*}
\item
When $\beta=0$ and $a,\alpha>0$, it reduces to the Cox-Ingersoll-Ross process with Poisson jumps, i.e.
\begin{equation*}
dr_{t}=b(c-r_{t})dt+\sigma\sqrt{r_{t}}dW_{t}+adN_{t},
\end{equation*}
where $N_{t}$ is a homogeneous Poisson process with intensity $\alpha$.
\item
When $c=0$ and $\sigma=0$, $N_{t}$ reduces to a Hawkes process with intensity $\lambda_{t}=\alpha+\beta r_{t}$, where
\begin{equation*}
dr_{t}=-br_{t}dt+adN_{t},
\end{equation*}
and it is easy to see that the intensity $\lambda_{t}$ indeed satisfies
\begin{equation*}
\lambda_{t}=\alpha+\beta\int_{0}^{t}ae^{-b(t-s)}N(ds),
\end{equation*}
which implies that $N_{t}$ is a classical linear Hawkes process with $\lambda(z)=\alpha+\beta z$ 
and $h(t)=ae^{-bt}$.
\end{enumerate}

It is easy to observe that $r_{t}$ is Markovian with generator
\begin{equation}
\mathcal{A}f(r)=bc\frac{\partial f}{\partial r}-br\frac{\partial f}{\partial r}
+\frac{1}{2}\sigma^{2}r\frac{\partial^{2}f}{\partial r^{2}}
+(\alpha+\beta r)[f(r+a)-f(r)].
\end{equation}

\subsection{Main Results}

In this section, we will summarize the main results of this paper. We will start with
conditional first and second moments of $r_{t}$ and then move onto the limit theorems, i.e. the law
of large numbers, central limit theorems and large deviations. Next, we show that there exists
a unique stationary probability measure for $r_{t}$ and we obtain the Laplace transform
of $r_{t}$ and $r_{\infty}$. Finally, we consider a short rate interest model.

The proofs will be given in Section \ref{PFs}.

The following proposition gives the formulas for the conditional first moment and second moment of
the Cox-Ingersoll-Ross process with Hawkes jumps.

\begin{proposition}\label{FirstandSecondMoments}
(i) For any $s>t$, we have the following conditional expectation,
\begin{equation}\label{FirstMoment}
\mathbb{E}[r_{s}|r_{t}]
=\frac{bc+a\alpha}{b-a\beta}-e^{-(b-a\beta)(s-t)}\left[\frac{bc+a\alpha}{b-a\beta}-r_{t}\right].
\end{equation}

(ii) For any $s>t$, we have the following conditional expectation,
\begin{align}\label{SecondMoment}
&\mathbb{E}[r_{s}^{2}|r_{t}]
\\
&=r_{t}^{2}e^{-2(b-a\beta)(s-t)}\nonumber
\\
&\qquad
+\left[(2bc+\sigma^{2}+2a\alpha+a^{2}\beta)\frac{bc+a\alpha}{2(b-a\beta)^{2}}+\frac{a^{2}\alpha}{2(b-a\beta)}\right]
[1-e^{-2(b-a\beta)(s-t)}]\nonumber
\\
&\qquad\qquad-(2bc+\sigma^{2}+2a\alpha+a^{2}\beta)\frac{bc+a\alpha}{(b-a\beta)^{2}}[e^{-(b-a\beta)(s-t)}-e^{-2(b-a\beta)(s-t)}]\nonumber
\\
&\qquad\qquad\qquad
+(2bc+\sigma^{2}+2a\alpha+a^{2}\beta)\frac{r_{t}}{b-a\beta}[e^{-(b-a\beta)(s-t)}-e^{-2(b-a\beta)(s-t)}].\nonumber
\end{align}
\end{proposition}

\begin{remark}
Let $a=0$ in \eqref{FirstMoment}, we get $\mathbb{E}[r_{s}|r_{t}]=c-e^{-b(s-t)}(c-r_{t})=r_{t}e^{-b(s-t)}+c(1-e^{-b(s-t)})$,
which recovers \eqref{FirstMomentCIR}.
Similarly, by letting $a=0$ in \eqref{SecondMoment}, we recover \eqref{SecondMomentCIR}.
\end{remark}

\begin{theorem}[Law of Large Numbers]\label{LLN}
For any $r_{0}:=r\in\mathbb{R}^{+}$,

(i)
\begin{equation}
\frac{1}{t}\int_{0}^{t}r_{s}ds\rightarrow\frac{bc+a\alpha}{b-a\beta},
\quad\text{in $L^{2}(\mathbb{P})$ as $t\rightarrow\infty$}.
\end{equation}

(ii)
\begin{equation}
\frac{N_{t}}{t}\rightarrow\frac{b(\alpha+\beta c)}{b-a\beta},
\quad\text{in $L^{2}(\mathbb{P})$ as $t\rightarrow\infty$}.
\end{equation}
\end{theorem}

\begin{theorem}[Central Limit Theorem]\label{CLT}
For any $r_{0}:=r\in\mathbb{R}^{+}$,

(i)
\begin{equation}
\frac{\int_{0}^{t}r_{s}ds-\frac{bc+a\alpha}{b-a\beta}t}{\sqrt{t}}
\rightarrow N\left(0,\frac{a^{2}\alpha(b-a\beta)+(a^{2}\beta+\sigma^{2})(bc+a\alpha)}{(b-a\beta)^{3}}\right),
\end{equation}
in distribution as $t\rightarrow\infty$.

(ii)
\begin{equation}
\frac{N_{t}-\frac{b(\alpha+\beta c)}{b-a\beta}t}{\sqrt{t}}
\rightarrow N\left(0,\frac{b^{3}a^{2}(\alpha+\beta c)+4\sigma^{2}b^{2}(bc+a\alpha)}{a^{2}(b-a\beta)^{3}}\right),
\end{equation}
in distribution as $t\rightarrow\infty$.
\end{theorem}

Before we proceed, recall that a sequence $(P_{n})_{n\in\mathbb{N}}$ of probability measures on a topological space $X$ 
satisfies the large deviation principle with rate function $I:X\rightarrow\mathbb{R}$ if $I$ is non-negative, 
lower semicontinuous and for any measurable set $A$, we have
\begin{equation}
-\inf_{x\in A^{o}}I(x)\leq\liminf_{n\rightarrow\infty}\frac{1}{n}\log P_{n}(A)
\leq\limsup_{n\rightarrow\infty}\frac{1}{n}\log P_{n}(A)\leq-\inf_{x\in\overline{A}}I(x).
\end{equation}
Here, $A^{o}$ is the interior of $A$ and $\overline{A}$ is its closure. 
We refer to Dembo and Zeitouni \cite{Dembo} and Varadhan \cite{VaradhanII} for general background of the theory and the applications
of large deviations.

\begin{theorem}[Large Deviation Principle]\label{LDP}
For any $r_{0}:=r\in\mathbb{R}^{+}$, 

(i)
$(\frac{1}{t}\int_{0}^{t}r_{s}ds\in\cdot)$ satisfies a large deviation principle
with rate function
\begin{equation}
I(x)=\sup_{\theta\leq\theta_{c}}\left\{\theta x-bcy(\theta)-\alpha(e^{ay(\theta)}-1)\right\},
\end{equation}
where for $\theta\leq\theta_{c}$, $y=y(\theta)$ is the smaller solution of 
\begin{equation}
-by+\frac{1}{2}\sigma^{2}y^{2}+\beta(e^{ay}-1)+\theta=0,
\end{equation}
and
\begin{equation}
\theta_{c}=by_{c}-\frac{1}{2}\sigma^{2}y_{c}^{2}-\beta(e^{ay_{c}}-1),
\end{equation}
where $y_{c}$ is the unique positive solution to the equation $b=\sigma^{2}y_{c}+\beta ae^{ay_{c}}$.

(ii) $(N_{t}/t\in\cdot)$ satisfies a large deviation principle
with rate function
\begin{equation}
I(x)=\sup_{\theta\leq\theta_{c}}\left\{\theta x-bcy(\theta)-\alpha(e^{ay(\theta)+\theta}-1)\right\},
\end{equation}
where for $\theta\leq\theta_{c}$, $y(\theta)$ is the smaller solution of
\begin{equation}
-by(\theta)+\frac{1}{2}\sigma^{2}y^{2}(\theta)+\beta(e^{ay(\theta)+\theta}-1)=0,
\end{equation}
and
\begin{equation}
\theta_{c}=\log\left(\frac{\sqrt{\sigma^{4}+a^{2}b^{2}+2a^{2}\sigma^{2}\beta}
-\sigma^{2}}{a^{2}\beta}\right)
-\frac{\sigma^{2}+ab-\sqrt{\sigma^{4}+a^{2}b^{2}+2a^{2}\sigma^{2}\beta}}{\sigma^{2}}.
\end{equation}
\end{theorem}

\begin{remark}
It is easy to see that when $c=0$ and $\sigma=0$, our results of Theorem \ref{LLN} (ii), Theorem \ref{CLT} (ii) and Theorem \ref{LDP} (ii)
are consistent with the law of large numbers and central limit theorem results for linear Hawkes process with exponential exciting
function as in Bacry et al. \cite{Bacry} and the large deviation principle as in Bordenave and Torrisi \cite{Bordenave}.
\end{remark}

\begin{proposition}\label{ergodiclemma}
Assume $b>a\beta$ and $2bc\geq\sigma^{2}$. Then, there exists a unique invariant probability measure for $r_{t}$. 
\end{proposition}

\begin{proposition}\label{Laplaceofr}
For any $\theta>0$, the Laplace transform of $r_{t}$ satisfies $\mathbb{E}[e^{-r_{t}}|r_{0}=r]
=e^{A(t)r+B(t)}$, where $A(t),B(t)$ satisfy the ordinary differential equations
\begin{equation}
\begin{cases}
A'(t)=-bA(t)+\frac{1}{2}\sigma^{2}A(t)^{2}+\beta(e^{aA(t)}-1),
\\
B'(t)=bcA(t)+\alpha(e^{aA(t)}-1),
\\
A(0)=-\theta, B(0)=0.
\end{cases}
\end{equation}
In particular, $\mathbb{E}[e^{-\theta r_{\infty}}]=e^{\int_{0}^{\infty}bcA(t)+\alpha(e^{aA(t)}-1)dt}$.
\end{proposition}

We can use $r_{t}$ as a stochastic model for short rate term structure.
We are interested to value a default-free discount bond paying one unit at time $T$, i.e.
\begin{equation}
P(t,T,r):=\mathbb{E}\left[e^{-\int_{t}^{T}r_{s}ds}\big|r_{t}=r\right].
\end{equation}

\begin{proposition}\label{shortrate}
(i) $P(t,T,r)=e^{A(t)r+b(t)}$, where $A(t),B(t)$ satisfy the following ordinary differential equations,
\begin{equation}
\begin{cases}
A'(t)-bA(t)+\frac{1}{2}\sigma^{2}A(t)^{2}+\beta(e^{aA(t)}-1)-1=0,
\\
B'(t)+bcA(t)+\alpha(e^{aA(t)}-1)=0,
\\
A(T)=B(T)=0.
\end{cases}
\end{equation}

(ii) We have the following asymptotic result,
\begin{equation}
\lim_{T\rightarrow\infty}\frac{1}{T}\log P(t,T,r)=bcx_{\ast}+\alpha(e^{ax_{\ast}}-1),
\end{equation}
where $x_{\ast}$ is the unique negative solution to the following equation,
\begin{equation}
-bx+\frac{1}{2}\sigma^{2}x^{2}+\beta(e^{ax}-1)-1=0.
\end{equation}
\end{proposition}

\begin{remark}
A natural way to generalize the Cox-Ingersoll-Ross process with Hawkes jumps is to allow the jump size to be random, i.e.
\begin{equation}
dr_{t}=b(c-r_{t})dt+\sigma\sqrt{r_{t}}dW_{t}+dJ_{t},
\end{equation}
where $J_{t}=\sum_{i=1}^{N_{t-}}a_{i}$,
and $a_{i}$ are i.i.d. positive random variables, independent of the past history and follows a probability distribution $Q(da)$. 
$N_{t}$ is a simple point process with intensity $\lambda_{t}=\alpha+\beta r_{t}$ at time $t>0$.
We assume that $a,b,c,\alpha,\beta,\sigma>0$, $b>\int_{\mathbb{R}^{+}}aQ(da)\beta$,
and $2bc\geq\sigma^{2}$.

We can write down the generator as
\begin{equation}
\mathcal{A}f(r)=bc\frac{\partial f}{\partial r}-br\frac{\partial f}{\partial r}+\frac{1}{2}\sigma^{2}r\frac{\partial^{2}f}{\partial r^{2}}
+(\alpha+\beta r)\int_{\mathbb{R}^{+}}[f(r+a)-f(r)]Q(da).
\end{equation}
All the results in this paper can be generalized to this model after a careful analysis.
\end{remark}

\begin{remark}
Another possibility to generalize the Cox-Ingersoll-Ross process with Hawkes jumps is to allow the jumps to follow a nonlinear
Hawkes process, i.e. $r_{t}$ satisfies the dynamics \eqref{dynamics}
and $N_{t}$ is a simple point process with intensity $\lambda(r_{t})$, where $\lambda(\cdot):\mathbb{R}^{+}\rightarrow\mathbb{R}^{+}$
is in general a nonlinear function. This can be considered as a generalization to the classical nonlinear Hawkes process
with exponential exciting function. Because of the nonlinearity, we will not be able to get a closed expression
in the limit for the limit theorems or a set of ordinary differential equations which are related to the Laplace transform
of the process.
\end{remark}

\section{Proofs}\label{PFs}

\begin{proof}[Proof of Proposition \ref{FirstandSecondMoments}]
(i) Taking expectations on both sides of \eqref{dynamics}, we have
\begin{equation}
d\mathbb{E}[r_{t}]=b(c-\mathbb{E}[r_{t}])dt+a(\alpha+\beta\mathbb{E}[r_{t}])dt,
\end{equation}
which implies that for any $s>t$, we have the following conditional expectation,
\begin{equation}
\mathbb{E}[r_{s}|r_{t}]
=\frac{bc+a\alpha}{b-a\beta}-e^{-(b-a\beta)(s-t)}\left[\frac{bc+a\alpha}{b-a\beta}-r_{t}\right].
\end{equation}

(ii) By It\^{o}'s formula, we have
\begin{equation}
d(r_{t}^{2})=2r_{t}[b(c-r_{t})dt+\sigma\sqrt{r_{t}}dW_{t}]
+\sigma^{2}r_{t}dt+2r_{t}adN_{t}+a^{2}dN_{t}.
\end{equation}
Taking expectations on both sides, we get
\begin{equation}
\frac{d\mathbb{E}[r_{t}^{2}]}{dt}=2bc\mathbb{E}[r_{t}]-2b\mathbb{E}[r_{t}^{2}]+\sigma^{2}\mathbb{E}[r_{t}]
+2a(\alpha\mathbb{E}[r_{t}]+\beta\mathbb{E}[r_{t}^{2}])+a^{2}\alpha+a^{2}\beta\mathbb{E}[r_{t}].
\end{equation}
This implies that
\begin{align}
&\mathbb{E}[r_{s}^{2}|r_{t}]e^{2(b-a\beta)s}-r_{t}^{2}e^{2(b-a\beta)t}
\\
&=(2bc+\sigma^{2}+2a\alpha+a^{2}\beta)\int_{t}^{s}e^{2(b-a\beta)u}\mathbb{E}[r_{u}|r_{t}]du
+a^{2}\alpha\int_{t}^{s}e^{2(b-a\beta)u}du\nonumber
\\
&=\left[(2bc+\sigma^{2}+2a\alpha+a^{2}\beta)\frac{bc+a\alpha}{2(b-a\beta)^{2}}+\frac{a^{2}\alpha}{2(b-a\beta)}\right]
[e^{2(b-a\beta)s}-e^{2(b-a\beta)t}]\nonumber
\\
&\qquad\qquad
-(2bc+\sigma^{2}+2a\alpha+a^{2}\beta)\frac{bc+a\alpha}{b-a\beta}\frac{e^{(b-a\beta)t}}{(b-a\beta)}[e^{(b-a\beta)s}-e^{(b-a\beta)t}]\nonumber
\\
&\qquad\qquad\qquad\qquad
+(2bc+\sigma^{2}+2a\alpha+a^{2}\beta)r_{t}\frac{e^{(b-a\beta)t}}{(b-a\beta)}[e^{(b-a\beta)s}-e^{(b-a\beta)t}],\nonumber
\end{align}
which yields \eqref{SecondMoment}.
\end{proof}

\begin{proof}[Proof of Theorem \ref{LLN}]
(i)
To prove the convergence in the $L^{2}(\mathbb{P})$ norm, we need to show that
\begin{align}
&\mathbb{E}\left(\frac{1}{t}\int_{0}^{t}r_{s}ds-\frac{bc+a\alpha}{b-a\beta}\right)^{2}
\\
&=\frac{1}{t^{2}}\mathbb{E}\left(\int_{0}^{t}r_{s}ds\right)^{2}-\frac{2}{t}\int_{0}^{t}\mathbb{E}[r_{s}]ds\cdot\frac{bc+a\alpha}{b-a\beta}
+\left(\frac{bc+a\alpha}{b-a\beta}\right)^{2}\rightarrow 0,\nonumber
\end{align}
as $t\rightarrow\infty$. From \eqref{FirstMoment} of Proposition \ref{FirstandSecondMoments}, 
it is clear that $\frac{1}{t}\int_{0}^{t}\mathbb{E}[r_{s}]ds\rightarrow
\frac{bc+a\alpha}{b-a\beta}$ as $t\rightarrow\infty$. Therefore, it suffices to show that $\frac{1}{t^{2}}\mathbb{E}\left(\int_{0}^{t}r_{s}ds\right)^{2}
\rightarrow\frac{bc+a\alpha}{b-a\beta}$ as $t\rightarrow\infty$. Applying \eqref{FirstMoment} of Proposition \ref{FirstandSecondMoments},
we get
\begin{align}
&\frac{1}{t^{2}}\mathbb{E}\left(\int_{0}^{t}r_{s}ds\right)^{2}
\\
&=\frac{2}{t^{2}}\iint_{0<s_{1}<s_{2}<t}\mathbb{E}[r_{s_{1}}\mathbb{E}[r_{s_{2}}|r_{s_{1}}]]ds_{1}ds_{2}\nonumber
\\
&=\frac{2}{t^{2}}\iint_{0<s_{1}<s_{2}<t}\frac{bc+a\alpha}{b-a\beta}\mathbb{E}[r_{s_{1}}]ds_{1}ds_{2}\nonumber
\\
&\qquad\qquad
-\frac{2}{t^{2}}\iint_{0<s_{1}<s_{2}<t}e^{-(b-a\beta)(s_{2}-s_{1})}\left[\frac{bc+a\alpha}{b-a\beta}\mathbb{E}[r_{s_{1}}]
-\mathbb{E}[r_{s_{1}}^{2}]\right]ds_{1}ds_{2}.\nonumber
\end{align}
From Proposition \ref{FirstandSecondMoments}, given $r_{0}=r$, $\mathbb{E}[r_{s_{1}}]$ and $\mathbb{E}[r_{s_{1}}^{2}]$
are uniformly bounded by some universal constant only depending on $r$, say $M(r)$. Therefore,
\begin{align}
&\left|\frac{2}{t^{2}}\iint_{0<s_{1}<s_{2}<t}e^{-(b-a\beta)(s_{2}-s_{1})}\left[\frac{bc+a\alpha}{b-a\beta}\mathbb{E}[r_{s_{1}}]
-\mathbb{E}[r_{s_{1}}^{2}]\right]ds_{1}ds_{2}\right|
\\
&\leq\frac{2}{t^{2}}M(r)\left[\frac{bc+a\alpha}{b-a\beta}+1\right]
\iint_{0<s_{1}<s_{2}<t}e^{-(b-a\beta)(s_{2}-s_{1})}ds_{1}ds_{2}\rightarrow 0,\nonumber
\end{align}
as $t\rightarrow\infty$. Again by \eqref{FirstMoment}, it is easy to check that
\begin{equation}
\frac{2}{t^{2}}\iint_{0<s_{1}<s_{2}<t}\frac{bc+a\alpha}{b-a\beta}\mathbb{E}[r_{s_{1}}]ds_{1}ds_{2}
\rightarrow\left(\frac{bc+a\alpha}{b-a\beta}\right)^{2},
\end{equation}
as $t\rightarrow\infty$. Hence, we proved the law of large numbers.

(ii) 
Observe that
$N_{t}-\int_{0}^{t}\lambda_{s}ds
=N_{t}-\alpha t-\beta\int_{0}^{t}r_{s}ds$
is a martinagle and
\begin{equation}
\mathbb{E}\left[\left(\frac{N_{t}-\int_{0}^{t}\lambda_{s}}{t}\right)^{2}\right]
=\frac{1}{t^{2}}\mathbb{E}\left[\int_{0}^{t}\lambda_{s}ds\right]
=\frac{\alpha}{t}+\frac{\beta}{t^{2}}\int_{0}^{t}\mathbb{E}[r_{s}]ds\rightarrow 0,
\end{equation}
as $t\rightarrow\infty$ by Proposition \ref{FirstandSecondMoments}. Therefore, we have
\begin{equation}
\frac{N_{t}}{t}-\alpha-\frac{\beta}{t}\int_{0}^{t}r_{s}ds\rightarrow 0,
\end{equation}
in $L^{2}(\mathbb{P})$ as $t\rightarrow\infty$ and the conclusion follows from (i).
\end{proof}

\begin{remark}
The $L^{2}$ convergence in Theorem \ref{LLN} implies the convergence in probability.
Indeed, the convergence in Theorem \ref{LLN} also holds almost surely by using Proposition \ref{ergodiclemma} and ergodic theorem.
For example, by ergodic theorem, $\frac{1}{t}\int_{0}^{t}r_{s}ds\rightarrow\mathbb{E}[r_{\infty}]$ almost surely as $t\rightarrow\infty$.
Let $\pi$ be the unique invariant probability measure of $r_{t}$, then, we have $\int\mathcal{A}f(r)\pi(dr)=0$ for any smooth function $f$.
Consider $f(r)=r$, we have $\int(bc-br+(\alpha+\beta r)a)\pi(dr)=0$ which implies that $\mathbb{E}[r_{\infty}]=\frac{bc+a\alpha}{b-a\beta}$.
Similarly, we can show that $\frac{N_{t}}{t}\rightarrow\frac{b(\alpha+\beta c)}{b-a\beta}$ as $t\rightarrow\infty$ almost surely.
Indeed, the a.s. convergence also follows by applying the large deviation principle and the Borel-Cantelli 
lemma. The limit can be identified as the unique zero of the corresponding rate function for the large deviations.
\end{remark}

\begin{proof}[Proof of Theorem \ref{CLT}]
(i)
Observe that $f(r_{t})-f(r_{0})-\int_{0}^{t}\mathcal{A}f(r_{s})ds$ is a martingale for $f(r)=Kr$, where 
$K$ is a constant to be determined. Let $f(r)=Kr$, then
\begin{equation}
\mathcal{A}f(r)=K\left[(a\beta-b)r+(\alpha a+bc)\right].
\end{equation}
Let us choose $K=\frac{1}{b-a\beta}>0$. Then, we have
\begin{equation}
\int_{0}^{t}r_{s}ds-\frac{bc+a\alpha}{b-a\beta}t=\left[f(r_{t})-f(r_{0})-\int_{0}^{t}\mathcal{A}f(r_{s})ds\right]-f(r_{t})+f(r_{0}).
\end{equation}
Since $f(r_{0})=Kr_{0}$ is fixed, $\frac{f(r_{0})}{\sqrt{t}}\rightarrow 0$ as $t\rightarrow\infty$.
Also, we have
\begin{equation}
\frac{\mathbb{E}[f(r_{t})]}{\sqrt{t}}=\frac{K\mathbb{E}[r_{t}]}{\sqrt{t}}
=\frac{K}{\sqrt{t}}\left\{\frac{bc+a\alpha}{b-a\beta}-e^{-(b-a\beta)t}\left[\frac{bc+a\alpha}{b-a\beta}-r_{0}\right]\right\}\rightarrow 0,
\end{equation}
as $t\rightarrow\infty$ by Proposition \ref{FirstandSecondMoments}. Therefore, $\frac{f(r_{t})}{\sqrt{t}}\rightarrow 0$
as $t\rightarrow\infty$ in probability. The quadratic variation of the martingale $f(r_{t})-f(r_{0})-\int_{0}^{t}\mathcal{A}f(r_{s})ds$
is the same as the quadratic variation of $f(r_{t})=\frac{1}{b-a\beta}r_{t}$, which is 
the same as the quadratic variation of $\frac{1}{b-a\beta}(aN_{t}+\int_{0}^{t}\sigma\sqrt{r_{s}}dW_{s})$, which is
$\frac{1}{(b-a\beta)^{2}}\left[a^{2}N_{t}+\sigma^{2}\int_{0}^{t}r_{s}ds\right]$.
By the law of large numbers in Theorem \ref{LLN}, we have
\begin{equation}
\frac{1}{t}\frac{1}{(b-a\beta)^{2}}\left[a^{2}N_{t}+\sigma^{2}\int_{0}^{t}r_{s}ds\right]
\rightarrow\frac{1}{(b-a\beta)^{2}}\left[a^{2}\alpha+\frac{(a^{2}\beta+\sigma^{2})(bc+a\alpha)}{b-a\beta}\right],
\end{equation}
as $t\rightarrow\infty$. Hence, by the usual central limit theorem for martingales, we conclude that
\begin{equation}
\frac{\int_{0}^{t}r_{s}ds-\frac{bc+a\alpha}{b-a\beta}t}{\sqrt{t}}
\rightarrow N\left(0,\frac{a^{2}\alpha(b-a\beta)+(a^{2}\beta+\sigma^{2})(bc+a\alpha)}{(b-a\beta)^{3}}\right),
\end{equation}
in distribution as $t\rightarrow\infty$.

(ii) From \eqref{dynamics}, we have
$N_{t}=\frac{r_{t}}{a}-\frac{r_{0}}{a}+\frac{bc}{a}t
-\frac{b}{a}\int_{0}^{t}r_{s}ds-\frac{\sigma}{a}\int_{0}^{t}\sqrt{r_{s}}dW_{s}$,
which implies that
\begin{align}
N_{t}-\frac{b(\alpha+\beta c)}{b-a\beta}
&=\frac{r_{t}}{a}-\frac{r_{0}}{a}-\frac{b}{a}\int_{0}^{t}\left(r_{s}-\frac{bc+a\alpha}{b-a\beta}\right)ds
-\frac{\sigma}{a}\int_{0}^{t}\sqrt{r_{s}}dW_{s}
\\
&=\frac{r_{t}}{a}-\frac{r_{0}}{a}-f(r_{t})+f(r_{0})\nonumber
\\
&\qquad\qquad\qquad
+\left[f(r_{t})-f(r_{0})-\int_{0}^{t}\mathcal{A}f(r_{s})ds\right]
-\frac{\sigma}{a}\int_{0}^{t}\sqrt{r_{s}}dW_{s},\nonumber
\end{align}
where $f(r)=-\frac{b}{a(b-a\beta)}$ and we know that
$\frac{1}{\sqrt{t}}\left[\frac{r_{t}}{a}-\frac{r_{0}}{a}-f(r_{t})+f(r_{0})\right]\rightarrow 0$ as $t\rightarrow\infty$
in probability by the arguments as in (i). Now, 
\begin{equation}
\left[f(r_{t})-f(r_{0})-\int_{0}^{t}\mathcal{A}f(r_{s})ds\right]
-\frac{\sigma}{a}\int_{0}^{t}\sqrt{r_{s}}dW_{s}
\end{equation}
is a martingale and it has the same quadratic variation as
\begin{equation}
-\frac{b}{b-a\beta}N_{t}-\frac{b\sigma}{a(b-a\beta)}\int_{0}^{t}\sqrt{r_{s}}dW_{s}-\frac{\sigma}{a}\int_{0}^{t}\sqrt{r_{s}}dW_{s},
\end{equation}
which has quadratic variation
$\frac{b^{2}}{(b-a\beta)^{2}}N_{t}+\frac{4\sigma^{2}b^{2}}{a^{2}(b-a\beta)^{2}}\int_{0}^{t}r_{s}ds$.
By law of large numbers, i.e. Theorem \ref{LLN}, we have
\begin{align}
&\frac{1}{t}\left[\frac{b^{2}}{(b-a\beta)^{2}}N_{t}+\frac{4\sigma^{2}b^{2}}{a^{2}(b-a\beta)^{2}}\int_{0}^{t}r_{s}ds\right]
\\
&\rightarrow\frac{b^{2}}{(b-a\beta)^{2}}\frac{b(a\alpha+\beta c)}{b-a\beta}+\frac{4\sigma^{2}b^{2}}{a^{2}(b-a\beta)^{2}}
\frac{bc+a\alpha}{b-a\beta},\nonumber
\end{align}
as $t\rightarrow\infty$. Hence, by the usual central limit theorem for martingales, we conclude that
\begin{equation}
\frac{N_{t}-\frac{b(\alpha+\beta c)}{b-a\beta}t}{\sqrt{t}}
\rightarrow N\left(0,\frac{b^{3}a^{2}(\alpha+\beta c)+4\sigma^{2}b^{2}(bc+a\alpha)}{a^{2}(b-a\beta)^{3}}\right),
\end{equation}
in distribution as $t\rightarrow\infty$.
\end{proof}

\begin{proof}[Proof of Theorem \ref{LDP}]
(i) Let $u(\theta,t,r):=\mathbb{E}[e^{\theta\int_{0}^{t}r_{s}ds}]$. Then, by Feynman-Kac formula, we have
\begin{equation}
\begin{cases}
\frac{\partial u}{\partial t}=bc\frac{\partial u}{\partial r}
-br\frac{\partial u}{\partial r}
+\frac{1}{2}\sigma^{2}r\frac{\partial^{2}u}{\partial r^{2}}
+(\alpha+\beta r)[u(\theta,t,r+a)-u(\theta,t,r)]+\theta ru=0,
\\
u(\theta,0,r)=1.
\end{cases}
\end{equation}
Let us try $u(\theta,t,r)=e^{A(t)r+B(t)}$, then $A(t)$ and $B(t)$ satisfy the following ordinary differential equations,
\begin{equation}
\begin{cases}
A'(t)=-bA(t)+\frac{1}{2}\sigma^{2}A(t)^{2}+\beta(e^{aA(t)}-1)+\theta,
\\
B'(t)=bcA(t)+\alpha(e^{aA(t)}-1),
\\
A(0)=B(0)=0.
\end{cases}
\end{equation}
It is easy to see that $\lim_{t\rightarrow\infty}A(t)=y$ where $y$ satisfies the equation
\begin{equation}\label{ythetaI}
-by+\frac{1}{2}\sigma^{2}y^{2}+\beta(e^{ay}-1)+\theta=0,
\end{equation}
if the equation has a solution and $\lim_{t\rightarrow\infty}A(t)=+\infty$ otherwise.

We claim that $y(\theta)$ is the smaller solution of the equation \eqref{ythetaI} for $\theta\leq\theta_{c}$, where
\begin{align}
\theta_{c}&=\max_{y\in\mathbb{R}^{+}}\left\{by-\frac{1}{2}\sigma^{2}y^{2}-\beta(e^{ay}-1)\right\}\label{thetacI}
\\
&=by_{c}-\frac{1}{2}\sigma^{2}y_{c}^{2}-\beta(e^{ay_{c}}-1),\nonumber
\end{align}
where $y_{c}$ is the unique positive solution to the equation $b=\sigma^{2}y_{c}+\beta ae^{ay_{c}}$.
This equation has a unique positive solution since $b>a\beta$.

Let us give more explanations here. The function $F(y):=-by+\frac{1}{2}\sigma^{2}y^{2}+\beta(e^{ay}-1)+\theta$
is convex and have two distinct solutions of $F(y)=0$ when $\theta<\theta_{c}$ and has a unique positive solution
when $\theta=\theta_{c}$. When $\theta<0$, $y(\theta)$ is the unique negative solution of $F(y)=0$ and when $0\leq\theta\leq\theta_{c}$,
$y(\theta)$ is the smaller non-negative solution of $F(y)=0$.

Hence, we have
\begin{equation}
\Gamma(\theta):=\lim_{t\rightarrow\infty}\frac{1}{t}\log u(\theta,t,r)=
\begin{cases}
bcy(\theta)+\alpha(e^{ay(\theta)}-1) &\text{if $\theta\leq\theta_{c}$}
\\
+\infty &\text{otherwise}
\end{cases}.
\end{equation}
Since $b>a\beta$, for $y$ being positive and sufficiently small in \eqref{thetacI},
we have $by-\frac{1}{2}\sigma^{2}y^{2}-\beta(e^{ay}-1)\sim by-\beta ay>0$ and thus $\theta_{c}>0$. 
Also $\Gamma(\theta)$ is differentiable for $\theta<\theta_{c}$ and
differentiating with respect to $\theta$ to \eqref{ythetaI}, we get
\begin{equation}
\frac{\partial y}{\partial\theta}=\frac{1}{b-\sigma^{2}y-\beta ae^{ay}}\rightarrow+\infty,
\end{equation}
as $\theta\uparrow\theta_{c}$, since $y\uparrow y_{c}$ as $\theta\uparrow\theta_{c}$. 
Therefore, we have the essential smoothness and by G\"{a}rtner-Ellis theorem (for the definition of essential smoothness and statement
of G\"{a}rtner-Ellis theorem, we refer to Dembo and Zeitouni \cite{Dembo}),
$(\frac{1}{t}\int_{0}^{t}r_{s}ds\in\cdot)$ satisfies a large deviation principle
with rate function
\begin{equation}
I(x)=\sup_{\theta\in\mathbb{R}}\left\{\theta x-bcy(\theta)-\alpha(e^{ay(\theta)}-1)\right\}.
\end{equation}

(ii) For a pair $(r_{t},N_{t})$, the generator is given by
\begin{equation}
\mathcal{A}f(r,n)=bc\frac{\partial f}{\partial r}-br\frac{\partial f}{\partial r}+\frac{1}{2}\sigma^{2}r\frac{\partial^{2}f}{\partial r^{2}}
+(\alpha+\beta r)[f(r+a,n+1)-f(r,n)].
\end{equation}
Let $u(t,r):=u(\theta,t,r):=\mathbb{E}[e^{\theta N_{t}}|r_{0}=r]$. Consider $f(t,r_{t},N_{t})=\mathbb{E}[e^{\theta N_{T}}|r_{t},N_{t}]$
and $f(t,r_{t},N_{t})_{t\leq T}$ is a martingale only if $\frac{\partial f}{\partial t}+\mathcal{A}f=0$ and $f(T,r_{T},N_{T})=e^{\theta N_{T}}$.
Let $f(t,r,n)=u(t,r)e^{\theta n}$ and make the time change $t\mapsto T-t$
to change the backward equation to the forward equation, we have
\begin{equation}
\begin{cases}
\frac{\partial u}{\partial t}=bc\frac{\partial u}{\partial r}-br\frac{\partial u}{\partial r}+\frac{1}{2}\sigma^{2}r\frac{\partial^{2}u}{\partial r^{2}}
+(\alpha+\beta r)[u(t,r+a)e^{\theta}-u(t,r)],
\\
u(0,r)\equiv 1.
\end{cases}
\end{equation}
Now, by trying $u(\theta,t,r)=e^{A(t)r+B(t)}$, we get
\begin{equation}
\begin{cases}
A'(t)=-bA(t)+\frac{1}{2}\sigma^{2}A^{2}(t)+\beta(e^{aA(t)+\theta}-1),
\\
B'(t)=bcA(t)+\alpha(e^{aA(t)+\theta}-1),
\\
A(0)=B(0)=0.
\end{cases}
\end{equation}
Hence, we have $\lim_{t\rightarrow\infty}A(t)=y(\theta)$, 
where $y(\theta)$ satisfies
\begin{equation}\label{ythetaII}
-by(\theta)+\frac{1}{2}\sigma^{2}y^{2}(\theta)+\beta(e^{ay(\theta)+\theta}-1)=0,
\end{equation}
if the above equation \eqref{ythetaII} has a solution and $+\infty$ otherwise. Similar to the arguments
in (i), $y(\theta)$ is the smaller solution of $\eqref{ythetaII}$ when $\theta\leq\theta_{c}$ and $+\infty$ otherwise.
$\theta_{c}$ is to be determined as the following.
We can rewrite the equation \eqref{ythetaII} as
\begin{equation}
e^{\theta}=\left(by-\frac{1}{2}\sigma^{2}y^{2}+\beta\right)\frac{1}{\beta}e^{-ay}.
\end{equation}
Let
\begin{align}\label{yandthetaIII}
\theta_{c}&=\log\max_{y\in\mathbb{R}^{+}}\left\{\left(by-\frac{1}{2}\sigma^{2}y^{2}+\beta\right)\frac{1}{\beta}e^{-ay}\right\}
\\
&=\log\left\{\left(by_{c}-\frac{1}{2}\sigma^{2}y_{c}^{2}+\beta\right)\frac{1}{\beta}e^{-ay_{c}}\right\}\nonumber
\\
&=\log\left(\frac{b-\sigma^{2}y_{c}}{a\beta}\right)-ay_{c}\nonumber
\\
&=\log\left(\frac{\sqrt{\sigma^{4}+a^{2}b^{2}+2a^{2}\sigma^{2}\beta}
-\sigma^{2}}{a^{2}\beta}\right)
-\frac{\sigma^{2}+ab-\sqrt{\sigma^{4}+a^{2}b^{2}+2a^{2}\sigma^{2}\beta}}{\sigma^{2}},\nonumber
\end{align}
where $y_{c}=\frac{\sigma^{2}+ab-\sqrt{(\sigma^{2}+ab)^{2}-2a\sigma^{2}(b-a\beta)}}{a\sigma^{2}}$.
Hence, we have
\begin{equation}
\Gamma(\theta):=\lim_{t\rightarrow\infty}\frac{1}{t}\log u(\theta, t,r)
=
\begin{cases}
bcy(\theta)+\alpha(e^{ay(\theta)+\theta}-1) &\text{if $\theta\leq\theta_{c}$}
\\
+\infty &\text{otherwise}
\end{cases}.
\end{equation}
Since $b>\beta a$, for $y$ being positive and sufficiently small in \eqref{yandthetaIII}, we have $(by-\frac{1}{2}\sigma^{2}y^{2}+\beta)\frac{1}{\beta}e^{-ay}
\sim(\frac{b}{\beta}y+1)(1-ay)\sim 1+(\frac{b}{\beta}-1)y>1$ and thus $\theta_{c}>0$. 
Also $\Gamma(\theta)$ is differentiable for $\theta<\theta_{c}$ and
differentiating with respect to $\theta$ to \eqref{ythetaII}, we get
\begin{equation}
\frac{\partial y}{\partial\theta}=\frac{\beta e^{ay+\theta}}{b-\sigma^{2}y-\beta ae^{ay+\theta}}\rightarrow+\infty,
\end{equation}
as $\theta\uparrow\theta_{c}$ since $y\uparrow y_{c}$ as $\theta\uparrow\theta_{c}$ and
by \eqref{yandthetaIII}, we have $e^{\theta_{c}}=\frac{b-\sigma^{2}y_{c}}{a\beta}e^{-ay_{c}}$.
Therefore, we have the essential smoothness and by G\"{a}rtner-Ellis theorem (for the definition of essential smoothness and statement
of G\"{a}rtner-Ellis theorem, we refer to Dembo and Zeitouni \cite{Dembo}),
$(N_{t}/t\in\cdot)$ satisfies a large deviation principle
with rate function
\begin{equation}
I(x)=\sup_{\theta\in\mathbb{R}}\left\{\theta x-bcy(\theta)-\alpha(e^{ay(\theta)+\theta}-1)\right\}.
\end{equation}
\end{proof}

\begin{proof}[Proof of Proposition \ref{ergodiclemma}]
The lecture notes \cite{Hairer} by Hairer gives the criterion for 
the existence and uniqueness of the invariant probability measure for Markov processes.
Suppose we have a jump diffusion process with generator $\mathcal{A}$. 
If we can find $u$ such that $u\geq 0$, $\mathcal{A}u\leq C_{1}-C_{2}u$ for some constants $C_{1},C_{2}>0$, 
then, there exists an invariant probability measure. In our problem, recall that
\begin{equation}
\mathcal{A}u(r)=bc\frac{\partial u}{\partial r}-br\frac{\partial u}{\partial r}+\frac{1}{2}\sigma^{2}r\frac{\partial^{2}u}{\partial r^{2}}
+(\alpha+\beta r)[u(r+a)-u(r)].
\end{equation}
Let us try $u(r)=r$ and choose $0<C_{2}<b-a\beta$, $C_{1}>\alpha a+bc$. Then, we have
\begin{align}
\mathcal{A}u+C_{2}u&=bc-br+\alpha a+\beta ar+C_{2}r
\\
&=(bc+\alpha a)+(\beta a-b+C_{2})r\nonumber
\\
&\leq bc+\alpha a\leq C_{1}.\nonumber
\end{align}
 
Next, we will prove the uniqueness of the invariant probability measure.
To get the uniqueness of the invariant probability measure, it is sufficient to prove that for any $x,y>0$, there exists some $T>0$ such that
$\mathcal{P}^{x}(T,\cdot)$ and $\mathcal{P}^{y}(T,\cdot)$ are not mutually singular. Here $\mathcal{P}^{x}(T,\cdot)=\mathbb{P}(r^{x}_{T}\in\cdot)$,
where $r^{x}_{T}$ is $r_{T}$ starting at $r_{0}=x$.
For any $x,y>0$, conditional on the event that $r_{t}^{x}$ and $r_{t}^{y}$ have no jumps during the time interval $(0,T)$, which has a positive probability, 
the law of $\mathcal{P}^{x}(T,\cdot)$ and $\mathcal{P}^{y}(T,\cdot)$ are absolutely continuous with respect to the Lebesgue measure 
on $\mathbb{R}^{+}$, which implies that $\mathcal{P}^{x}(T,\cdot)$ and $\mathcal{P}^{y}(T,\cdot)$ are not mutually singular. 
\end{proof}

\begin{proof}[Proof of Proposition \ref{Laplaceofr}]
By Kolmogorov equation, $u(t,r)=\mathbb{E}[e^{-\theta r_{t}}|r_{0}=r]$ satisfies
\begin{equation}
\begin{cases}
\frac{\partial u}{\partial t}=bc\frac{\partial u}{\partial r}-br\frac{\partial u}{\partial r}
+\frac{1}{2}\sigma^{2}r\frac{\partial^{2}u}{\partial r^{2}}+(\alpha+\beta r)[u(t,r+a)-u(t,r)],
\\
u(0,r)=e^{-\theta r}.
\end{cases}
\end{equation}
Now, try $u(t,r)=e^{A(t)r+B(t)}$, we get the desired results.
\end{proof}

\begin{proof}[Proof of Proposition \ref{shortrate}]
(i) By Feynman-Kac formula, $P(t,T,r)$ satisfies the following integro-partial differential equation,
\begin{equation}
\begin{cases}
\frac{\partial P}{\partial t}+bc\frac{\partial P}{\partial r}
-br\frac{\partial P}{\partial r}+\frac{1}{2}\sigma^{2}r\frac{\partial^{2}P}{\partial r^{2}}
\\
\qquad\qquad\qquad
+(\alpha+\beta r)[P(t,T,r+a)-P(t,T,r)]-rP(t,T,r)=0,
\\
P(T,T,r)=1.
\end{cases}
\end{equation}
Let us try $P(t,T,r)=e^{A(t)r+B(t)}$. We get
\begin{equation}
\begin{cases}
A'(t)-bA(t)+\frac{1}{2}\sigma^{2}A(t)^{2}+\beta(e^{aA(t)}-1)-1=0,
\\
B'(t)+bcA(t)+\alpha(e^{aA(t)}-1)=0,
\\
A(T)=B(T)=0.
\end{cases}
\end{equation}

(ii)
By using the same arguments as in the proof of Theorem \ref{LDP}, we have the following asymptotic result,
\begin{equation}
\lim_{T\rightarrow\infty}\frac{1}{T}\log P(t,T,r)=bcx_{\ast}+\alpha(e^{ax_{\ast}}-1),
\end{equation}
where $x_{\ast}$ is the unique negative solution to the following equation,
\begin{equation}
-bx+\frac{1}{2}\sigma^{2}x^{2}+\beta(e^{ax}-1)-1=0.
\end{equation}
\end{proof}

\section*{Acknowledgements}

The author is supported by NSF grant DMS-0904701, DARPA grant and MacCracken Fellowship at New York University.
The author is very grateful to an anonymous referee for the helpful comments and suggestions.

\end{document}